\documentclass[11pt]{article}
\usepackage{geometry, hyperref, verbatim, enumerate,amsthm, amsmath, amssymb, amscd, mathrsfs, color, mathabx}

\title{Fuchs' problem for 2-groups}

\author{Eric Swartz\footnote{Department of Mathematics, William \& Mary, Williamsburg, VA, USA.\newline E-mail: easwartz@wm.edu}\\
Nicholas J.\ Werner\footnote{Department of Mathematics, Computer and Information Science, State University of New York College at Old Westbury, Old Westbury, NY, USA.\newline E-mail: wernern@oldwestbury.edu}
}

\newcommand{\Z}{\mathbb{Z}}
\newcommand{\F}{\mathbb{F}}

\newcommand{\rchar}{\text{char}}

\newcommand{\Ker}{\text{Ker}\,}

\newcommand{\bfi}{\mathbf{i}}
\newcommand{\bfj}{\mathbf{j}}
\newcommand{\ep}{\varepsilon}

\newcommand{\asum}[1]{$1+x_{#1}+x_{#1}^2+x_{#1}^3$}
\newcommand{\bsum}[1]{$1+x_{#1}+x_{#1}^2+x_{#1}^3\texttt{f6}$}
\newcommand{\f}{\texttt{f6}}

\newcommand{\les}{\leqslant}
\newcommand{\ges}{\geqslant}

\newcommand{\PL}{d}

\numberwithin{equation}{section}

\theoremstyle{definition}
\theoremstyle{plain}\newtheorem{Lem}[equation]{Lemma}
\theoremstyle{plain}\newtheorem{Prop}[equation]{Proposition}
\theoremstyle{plain}\newtheorem{Thm}[equation]{Theorem}
\theoremstyle{plain}\newtheorem{Cor}[equation]{Corollary}
\theoremstyle{plain}
\theoremstyle{remark}
\theoremstyle{remark}
\theoremstyle{definition}\newtheorem{Ex}[equation]{Example}
\theoremstyle{definition}
\theoremstyle{definition}\newtheorem{Ques}[equation]{Question}
\theoremstyle{definition}\newtheorem{mydef}[equation]{Definition}

\newtheorem{innercustomthm}{Theorem}
\newenvironment{customthm}[1]
  {\renewcommand\theinnercustomthm{#1}\innercustomthm}
  {\endinnercustomthm}

\begin{document}

\maketitle
\thispagestyle{empty}

\abstract{Nearly $60$ years ago, L\'{a}szl\'{o} Fuchs posed the problem of determining which groups can be realized as the group of units of a commutative ring.  To date, the question remains open, although significant progress has been made.  Along this line, one could also ask the more general question as to which finite groups can be realized as the group of units of a finite ring.  In this paper, we consider the question of which $2$-groups are realizable as unit groups of finite rings, a necessary step toward determining which nilpotent groups are realizable.  We prove that all $2$-groups of exponent $4$ and exponent $2$ are realizable in characteristic $2$, and we prove that many $2$-groups with exponent $4$ and nilpotency class $3$ are realizable in characteristic $2$.  On the other hand, we provide an example of a $2$-group with exponent $4$ and nilpotency class $4$ that is not realizable in characteristic $2$.  Moreover, while some groups of exponent greater than $4$ are realizable as unit groups of rings, we prove that any $2$-group with a self-centralizing element of order $8$ or greater is never realizable in characteristic $2^m$, and consequently any indecomposable, nonabelian group with a self-centralizing element of order $8$ or greater cannot be the group of units of a finite ring.}

\section{Introduction}\label{Intro section}

The purpose of this paper is to describe which finite 2-groups occur as the unit group of a finite ring. Throughout, all rings are associative and unital. For a ring $R$, $R^\times$ denotes the group of units of $R$. Given a group $G$, we say that $G$ is \textit{realizable} if there exists a ring $R$ such that $R^\times = G$. Determining whether a group or family of groups is realizable has come to be called \textit{Fuchs' problem} after L\'{a}szl\'{o} Fuchs, who posed the question of characterizing the groups that can occur as the group of units of a commutative ring \cite[Problem 72, p.\ 299]{Fuchs}. 

To date, no complete answer has been given to Fuchs' problem, although many partial answers or modifications have been produced. In \cite{Gilmer}, Gilmer determined all finite commutative rings $R$ such that $R^\times$ is cyclic; more recently, Dol\v{z}an \cite{DolzanNilpotent} characterized finite rings whose group of units is nilpotent (thus correcting an erroneous solution to this problem given in \cite[Cor.\ XXI.10]{McDonald}). All finite realizable groups of odd order were described by Ditor in \cite{Ditor}. In the past decade, Davis and Occhipinti determined all realizable finite simple groups \cite{DavisOcchSimple}, as well as all realizable alternating and symmetric groups \cite{DavisOcchAltSymm}. During the same time period, Chebolu and Lockridge solved Fuchs' problem for dihedral groups \cite{ChebLockDihedral} and recently began the study of the problem for $p$-groups \cite{ChebLockPGroup}.  Several other recent articles have investigated realizable groups in the traditional commutative setting \cite{ChebLockFields, ChebLockIndecomp, ChebLockHowMany, DelCorso, DelCorsoDvornFiniteGroups, DelCorsoDvornFuchs}.

A natural first generalization of Fuchs' original problem, further motivated by Dol\v{z}an's work \cite{DolzanNilpotent}, is to ask which nilpotent groups can be realized as the group of units of a finite ring. Given that a finite nilpotent group $G$ is the direct product of its Sylow subgroups, by the work of Ditor \cite{Ditor}, sufficient conditions for $G$ to be realizable can be obtained by studying which 2-groups are realizable as the group of units of a finite ring.  This is the aim of the present paper.

To further put the work of this paper in context, the recent paper \cite{ChebLockPGroup} began the study of the realizability of $2$-groups as unit groups of rings (possibly having characteristic $0$).  It was shown that, for a fixed positive integer $n$, there are only finitely many abelian $2$-groups of rank $n$ that are realizable in characteristic $2$; see \cite[Theorem 1.2]{ChebLockPGroup}.  Moreover, the $2$-groups that have a cyclic subgroup of index $2$ or are generalized quaternion are classified; see \cite[Theorem 1.4]{ChebLockPGroup}.


While we are not able to provide a complete classification of realizable 2-groups, we will present partial results that apply to large classes of 2-groups. Not every 2-group is realizable, and the exponent of the group turns out to be a significant factor in determining realizability. Indeed, all of our most significant theorems, which are stated below, involve conditions on the exponent of the group.

Before stating our main result, we present the following somewhat technical definition. For $a, b \in G$, $[a,b] = a^{-1}b^{-1}ab$, and if $a, b, c \in G$, then their triple commutator is $[a, b, c] := [[a, b], c]$.

\begin{mydef}\label{def:goodseq}
We define a \textit{good generating sequence} $X = (x_1,\dots, x_n)$ to be a generating set for $G$ such that every element $g \in G$ can be written uniquely as
\[ g = \prod_{i=1}^n x_i^{\ep_i},\]
where each $\ep_i \in \{0,1\}$, and the sequence further satisfies:
\begin{itemize}
    \item[(1)] for $i < j$, $[x_j, x_i] \in X \cup \{1\}$ and, if $[x_j, x_i] \notin Z(G)$, then:
        \begin{itemize}
            \item[(i)] $[x_j, x_i] = x_m$ for some $m \in \{1,\dots, n\}$ such that $i < m < j$;
            \item[(ii)] if $x_i^2 \notin Z(G)$, $x_i^2 = \prod_{m \in M} x_m$ for some $M \subseteq \{1,\ldots,n\}$, and $x_\ell \notin Z(G)$ for some $\ell \in M$, then $\ell < j$;
            \item[(iii)] if $x_j^2 \notin Z(G)$, $x_j^2 = \prod_{m \in M} x_m$ for some $M \subseteq \{1,\ldots,n\}$, and $x_\ell \notin Z(G)$ for some $\ell \in M$, then $i < \ell$;
        \end{itemize}
    \item[(2)] for $i < j < k$, if $[x_k, x_i, x_j] = 1$.  
\end{itemize} 
\end{mydef}

\begin{Thm}\label{thm:exp4nil2} 
Let $G$ be a $2$-group of exponent $4$.  If $G$ has nilpotency class at most $2$ or if $G$ has nilpotency class $3$ and a good generating sequence, then $G$ is realizable as the group of units of a ring with characteristic $2$.
\end{Thm}

Higman \cite{Higman} and Sims \cite{Sims} determined lower and upper bounds, respectively, on the number of isomorphism classes of finite 2-groups.  It follows from their work that

\[ \frac{\log(\# \text{ of groups of order }2^n \text{ with exponent }4 \text{ and nilpotency class } 2)}{\log(\# \text{ of groups of order } 2^n)} \to 1 \]
as $n \to \infty$.  Thus, Theorem \ref{thm:exp4nil2} implies that many, perhaps even most, finite 2-groups are realizable in characteristic 2.  For example, a calculation in GAP \cite{GAP} shows that exactly 8789818 out of 10494213 groups of order $512$ satisfy the hypotheses of Theorem \ref{thm:exp4nil2}, which comes out to about 83.76\% of all groups of order $512$.  

On the other hand, not all groups with exponent $4$ are realizable in characteristic $2$: Section \ref{sect:counterexample} is devoted to proving that SmallGroup(64,34) in GAP \cite{GAP}, which has presentation
\begin{equation*}
\langle x_1, x_2 : x_1^4=x_2^2 = (x_1x_2)^4= (x_1^2x_2)^4 = 1 \rangle,
\end{equation*}
is not realizable in characteristic 2.

For 2-groups of exponent at least 8, the situation is perhaps even more nebulous. In Section \ref{Prelim section}, we will prove that if $G$ is a nonabelian, indecomposable 2-group and $G=R^\times$ for a finite ring $R$, then the characteristic of $R$ must be $2^m$ for some $m \ges 1$ (Proposition \ref{Ring char prop}). This allows us to narrow our focus to rings of characteristic $2^m$. In some cases, we are able to prove that 2-groups of large exponent are not realizable in characteristic 2, or not realizable at all.

\begin{Thm}\label{Exponent upper bound}
Let $n \ges 1$, let $G$ be a 2-group of order $2^n$ that is realizable in characteristic $2^m$, and let $L= \lceil \log_2(n+1) \rceil$. Then, the exponent of $G$ is at most $2^{L+m-1}$.
\end{Thm}

\begin{Thm}\label{Exp at least 8 cor 1}
Let $G$ be a finite $2$-group. Assume that there exists $a \in G$ such that $|a| \ges 8$ and $C_G(a) = \langle a \rangle$. Then, $G$ is not realizable in characteristic $2^m$ for any $m \ges 1$.
\end{Thm}

Theorem \ref{Exp at least 8 cor 1} can be applied to some well known families of groups of order $2^n$, including cyclic groups (for $n \ges 3$), generalized quaternion groups (for $n \ges 4$), and quasidihedral groups (for $n \ges 4$). Hence, we recover the result from \cite{ChebLockPGroup} that none of these groups is realizable in characteristic $2^m$, and the latter two families---being nonabelian and indecomposable---are not realized by any finite ring. By contrast, there exist realizable 2-groups of arbitrarily large exponent. For instance, when $m \ges 3$, the unit group of the integers mod $2^m$ is isomorphic to $C_{2^{m-2}} \times C_2$---which has exponent $2^{m-2}$---and thus this group is realizable in characteristic $2^m$.

The paper is organized as follows. In Section \ref{Prelim section}, we set up our basic notation and translate the question of realizing a 2-group $G$ to the study of residue rings of group rings over $G$. Section \ref{Exponent 4 section} is devoted to the proof of Theorem \ref{thm:exp4nil2}, Section \ref{sect:counterexample} contains an example of a group with exponent $4$ that is not realizable in characteristic $2$, and Section \ref{Large exponent section} focuses on Theorems \ref{Exponent upper bound} and \ref{Exp at least 8 cor 1}. Finally, in Section \ref{Example section}, we collect a number of intriguing examples and questions about the existence of 2-groups satisfying certain properties (e.g.,\ Question \ref{Realizable, but not in char 2}: Does there exist an indecomposable 2-group that is realizable in characteristic $2^m$ for some $m \ges 2$, but is not realizable in characteristic 2?).  While we are able to provide many answers, some examples raise more questions than they answer, which will hopefully inspire future investigation.

\section{Preliminaries}\label{Prelim section}

We begin by recalling some standard notation and terminology. For any positive integer $n$, $\Z_n$ denotes the ring of integers mod $n$. For a prime power $p^n$, $\F_{p^n}$ is the finite field with $p^n$ elements. A group is \textit{indecomposable} if it is not isomorphic to a direct product of two nontrivial groups. Likewise, a ring is \textit{indecomposable} if it is not isomorphic to a direct product of two nontrivial rings. 
The characteristic of a ring $R$ is denoted by $\rchar(R)$.

For a finite group $G$ and ring $R$, $R[G]$ will denote the group ring of $G$ over $R$. The elements of $R[G]$ are sums of the form $\sum_{g \in G} \lambda_g g$, where each $\lambda_g \in R$. These sums are added componentwise, and are multiplied by using the rule $(\lambda_g g)\cdot(\lambda_h h) = \lambda_g \lambda_h gh$ and extending linearly. Usually, our group rings will be over $\Z_m$ for some $m \ges 2$. In this situation, $G$ is a subgroup of $(\Z_m[G])^\times$, and we will show shortly that much of the work needed to decide whether $G$ is realizable in characteristic $m$ comes down to considering residue rings of $\Z_m[G]$.

Next, we collect some elementary, but extremely useful, observations about finite rings and their unit groups in characteristic $m$.

\begin{Lem}\label{Units of Zm}
Let $G$ be a finite group and let $m \ges 2$. If $G$ is realizable in characteristic $m$, then $\Z_{m}^\times \les Z(G)$. Hence, there are only finitely many possible characteristics in which $G$ could be realized.
\end{Lem}
\begin{proof}
Let $R$ be a ring of characteristic $m$ such that $R^\times = G$. Then, $R$ contains a copy of the ring $\Z_m$ that is central in $R$, and hence $\Z_{m}^\times \les Z(G)$. The second claim is true because for a fixed finite group $G$, it is possible to find $k$ such that $|\Z_{m}^\times| > |G|$ for all $m \ges k$.
\end{proof}

As noted in \cite[Lem.\ 6]{DavisOcchSimple} and \cite[Prop.\ 2.2]{ChebLockDihedral}, if $R$ is a ring of characteristic $m$ such that $R^\times \cong G$, then the natural embedding $G \to R$ extends to a ring homomorphism $\phi: \Z_{m}[G] \to R$. The image of $\phi$ is a (possibly proper) subring of $R$ that also has group of units isomorphic to $G$. Hence, we obtain the following lemma, which is the basis for much of our subsequent work.

\begin{Lem}\label{Group ring image}\cite[Lem.\ 6]{DavisOcchSimple}, \cite[Prop.\ 2.2]{ChebLockDihedral} Let $G$ be a finite group and let $m \ges 2$. If $G$ is realizable in characteristic $m$, then there exists a two-sided ideal $I$ of $\Z_{m}[G]$ such that $(\Z_{m}[G]/I)^\times \cong G$. 
\end{Lem}

In the case of 2-groups, we have the following theorem of Dol\v{z}an that gives a broad description of those finite rings $R$ for which $R^\times$ is a 2-group.

\begin{Thm}\label{Dolzan thm} \cite[Cor.\ 4.4]{DolzanUnits}
Let $G$ be a finite 2-group, and let $R$ be a finite ring such that $R^\times \cong G$. Then, $R$ is a direct product of rings, where every direct factor is isomorphic to one of the following rings:
\begin{enumerate}
\item The field $\F_{2^k+1}$, where $2^k+1$ is a power of a prime.
\item The field $\F_2$.
\item A local 2-ring.
\item The ring $\left\{ \left[ \begin{smallmatrix} a & b\\ 0 & c \end{smallmatrix} \right] : a, c \in \F_2, b \in V \right\}$, where $V$ is a vector space over $\F_2$.
\item A 2-ring such that its group of units $G$ can be written as a product of its two proper subgroups, either $e+\overline{e}G$ and $\overline{e}+eG$, or $e+G\overline{e}$ and $\overline{e}+Ge$, for a nontrivial idempotent $e$.
\end{enumerate}
\end{Thm}

For products of nonabelian, indecomposable 2-groups, we can say more.

\begin{Prop}\label{Ring char prop}
Let $G = \prod_{i=1}^k G_i$, where $k \ges 1$ and each $G_i$ is a nonabelian, indecomposable 2-group. If $G$ is the group of units of a finite ring $R$, then $\rchar(R) = 2^m$ for some $m \ges 1$.
\end{Prop}
\begin{proof}
Assume $R$ is such that $R^\times \cong G$. By Theorem \ref{Dolzan thm}, we may express $R$ as a direct product $R \cong \prod_{j=1}^t R_j$, where $t \ges 1$ and each $R_j$ one of the rings listed in Theorem \ref{Dolzan thm}. Note that $G \cong \prod_{j=1}^t R_j^\times$. If some $R_j$ is isomorphic to a field $\F_{2^k+1}$, then $R_j^\times \cong C_{2^k}$, which violates the assumption that each direct factor of $G$ is nonabelian. Hence, each ring $R_j$ must be one the types listed in parts 2--5 of Theorem \ref{Dolzan thm}, and any such ring is a 2-ring. Therefore, $R$ must have characteristic $2^m$ for some $m \ges 1$.
\end{proof}

For a direct product of nonabelian, indecomposable 2-groups to be realizable, it is sufficient that each direct factor be realizable. If such a group $G$ is realized as the unit group of $R$, then by Proposition \ref{Ring char prop} the characteristic of $R$ is $2^m$ for some $m \ges 1$. But, Lemma \ref{Group ring image} shows that we may also assume $R$ is a residue ring of the group ring $\Z_{2^m}[G]$. Hence, it is beneficial to study these group rings more closely.  

\begin{Lem}\label{Group ring is local}
Let $G$ be a 2-group. Then, for all $m \ges 1$, the group ring $\Z_{2^m}[G]$ is local, with unique maximal ideal $M = \{\sum_i \lambda_i g_i : \sum_i \lambda_i \equiv 0 \mod 2\}$ and residue field $\F_2$. Consequently, the unit group of $\Z_{2^m}[G]$ is $\Z_{2^m}[G] \setminus M = 1 + M$.
\end{Lem}
\begin{proof}
Let $A = \Z_{2^m}[G]$ and let $J$ be the Jacobson radical of $A$. By composing the augmentation homomorphism $A \to \Z_{2^m}$ with reduction modulo 2, we obtain a surjective ring homomorphism $A \to \Z_2$ with kernel $M$. The ideal $M$ is maximal, so $J \subseteq M$. To complete the proof, it suffices to show that $M \subseteq J$.

Let $\pi: A \to A/2A$ be the canonical quotient map. Then, $A/2A \cong \Z_2[G]$, and $\pi(M)$ lies in the augmentation ideal of $A/2A$. This augmentation ideal is nilpotent \cite[Thm.\ 6.3.1]{PolcinoSehgal}, so $\pi(M^k) = \pi(M)^k = (0)$ in $A/2A$ for some $k \ges 1$. Hence, $M^k \subseteq 2A$, which means that $M^{kn} = (0)$ in $A$ for some $n \ges 1$.  Thus, $M$ is a nilpotent ideal, and $M \subseteq J$ by \cite[Prop.\ IV.7]{McDonald}. Finally, if $u \in A \setminus M$, then $u = 1+a$ for some $a \in M$, and hence $u$ is a unit in $A$.
\end{proof}

We summarize the results of this section in the following theorem.

\begin{Thm}\label{R=2G thm}
Let $G$ be a finite 2-group that is realizable in characteristic $2^m$ for some $m \ges 1$. Then, there exists a ring $R$ such that $R^\times \cong G$ and the following hold:
\begin{enumerate}[(i)]
\item $R$ is a residue ring of $\Z_{2^m}[G]$.
\item $R$ is local.
\item $|R|=2|G|$.
\item $G=1+M$, where $M$ is the maximal ideal of $R$.
\end{enumerate}
\end{Thm}
\begin{proof}
Property (i) follows from Lemma \ref{Group ring image}. Let $J$ be the Jacobson radical of $\Z_{2^m}[G]$ and let $U = 1+J$ be the unit group of $\Z_{2^m}[G]$. Since the residue field of $\Z_{2^m}[G]$ is $\F_2$, we have $|U|=|J|=\frac{1}{2}|\Z_{2^m}[G]|$. Let $\pi: \Z_{2^m}[G] \to R$ be the quotient map. Then, $\Ker \pi$ is contained in $J$, so $\pi(J)$ is the unique maximal ideal of $R$ and $G = R^\times = \pi(U)$. The stated properties of $R$ follow.
\end{proof}

\section{Realizable groups of exponent 4}\label{Exponent 4 section}

Proposition \ref{Ring char prop} and Theorem \ref{R=2G thm} indicate that to see if a 2-group is realizable in characteristic $2^m$, we should study $\Z_{2^m}[G]$ and its residue rings. Unsurprisingly, this is easiest to do when $m=1$, and in this section we will prove the first of several theorems on the realizability of 2-groups in characteristic 2. Among these is Theorem \ref{thm:exp4nil2}, which shows that any finite 2-group of exponent 4 is realizable in characteristic 2.

Our approach is motivated by the relationship between $G$ and the unit group of $\Z_2[G]$. Recall that if $U$ is a group such that $G \les U$, then a normal complement of $G$ in $U$ is a normal subgroup $N$ of $U$ such that $U=GN$ and $G \cap N = \{1\}$.

\begin{Lem}\label{Norm comp lemma}
Let $G$ be a finite 2-group, and let $U = (\Z_2[G])^\times$. 
\begin{enumerate}[(1)]
\item Let $I$ be a two-sided ideal of $\Z_2[G]$. Then, $1+I$ is a normal subgroup of $U$.
\item Assume that $G$ has a normal complement $N$ in $U$. If $1+N$ is a two-sided ideal of $\Z_2[G]$, then $G$ is realizable in characteristic 2.
\end{enumerate}
\end{Lem}
\begin{proof}
(1) Let $M$ be the maximal ideal of $\Z_2[G]$. Then, $I \subseteq M$, and $U=1+M$ by Lemma \ref{Group ring is local}, so $1+I \subseteq U$. Moreover, since $I$ is an ideal, $1+I$ is closed under multiplication and under conjugation by elements of $U$. Hence, $1+I \lhd U$.

(2) Assume that $I:=1+N$ is a two-sided ideal of $\Z_2[G]$ and let $R = \Z_2[G]/I$. Then, $U$ surjects onto $R^\times$, and the kernel of this map is $N$, so $G \cong U/N \cong R^\times$.
\end{proof}

If a normal complement of $G$ in $(\Z_{2}[G])^\times$ always produced a corresponding two-sided ideal of $\Z_2[G]$, then determining whether $G$ is realizable could be answered solely by studying $(\Z_2[G])^\times$. Unfortunately, this is not the case; it is possible for a 2-group to have a normal complement in $(\Z_2[G])^\times$ and not be realizable in characteristic 2. See Question \ref{Norm comp, but not realizable}, where this is discussed for the cyclic group $C_8$ and also a nonabelian group of order 16. However, in the special case where $G$ has exponent 4, we are able to prove that a normal complement of $G$ in $(\Z_2[G])^\times$ exists and can be translated into a two-sided ideal of $\Z_2[G]$.

Recall that a group $G$ is said to be \textit{nilpotent} of class $n$ if there exists a central series of length $n$, that is, if there exists a normal series 
\begin{equation*}
{1} = G_0 \lhd G_1 \lhd  \cdots \lhd G_n = G
\end{equation*}
such that $G_{i+1}/G_i \les Z(G/G_i)$. Moreover, a group has \textit{nilpotency class $2$} if $G/Z(G)$ is abelian.  The next lemma contains properties of groups of nilpotency class $2$, which we will use to prove that all groups of exponent $4$ and nilpotency class $2$ are realizable. In what follows, $[a,b]$ denotes the commutator of two elements $a, b$ in a group $G$.

\begin{Lem}\cite[Proposition VI.1.k]{Schenkman}
\label{lem:nil2groups}
Let $G$ be a group of nilpotency class $2$, let $a,b$ be arbitrary elements of $G$, and let $m,n$ be arbitrary integers.  Then, the following hold:
\begin{itemize}
 \item[(i)] $[a^m, b^n] = [a,b]^{mn}$,
 \item[(ii)] $(ab)^n = a^n b^n [b,a]^{\binom{n}{2}}$.
\end{itemize}
\end{Lem}

\begin{Lem}
 \label{lem:nil2exp4}
 Let $G$ be a $2$-group of nilpotency class $2$ and exponent $4$.  Then, the following hold:
 \begin{itemize}
  \item[(i)] for all $a,b \in G$, $[b,a]^2 = 1$,
  \item[(ii)] for all $a \in G$, $a^2 \in Z(G)$.
 \end{itemize}
\end{Lem}

\begin{proof}
 We will show (i) first.  Let $a, b \in G$.  Then, since $G$ has exponent $4$ and using Lemma \ref{lem:nil2groups}(ii), we have
 \[1 = (ab)^4 = a^4b^4[b,a]^6 = [b,a]^2. \]
 Finally, since $[b,a^2] = [b,a]^2 = 1$ for all $a,b \in G$ by (i) and Lemma \ref{lem:nil2groups}(i), we have that $a^2 \in Z(G)$ for all $a \in G$.
\end{proof}

In general, the unit group of $\Z_m[G]$ is decomposable \cite[Chap.\ 8]{PolcinoSehgal}, and we have
\begin{equation*}
(\Z_m[G])^\times \cong \Z_m^\times \times U_1(G),
\end{equation*}
where $U_1(G) = \{\sum_i \lambda_i g_i : \sum_i \lambda_i = 1\}$ is the subgroup of units with coefficient sum 1. Of course, when $m=2$, $(\Z_2[G])^\times = U_1(G)$ because $\Z_2^\times$ is trivial. If $m=p$ is a prime, then $U_1(G)$ is known as the \textit{mod $p$ envelope} of $G$. The following result by Moran and Tench provides a sufficient condition for a $p$-group $G$ to have a normal complement in the mod $p$ envelope of $G$ (and hence in the group of units of $\Z_p[G]$) and is crucial to proving the realizability of groups with exponent $4$. (For further work regarding whether a finite $p$-group $G$ has a normal complement in $U_1(G)$, see \cite{Ivory, Johnson}.)


\begin{Lem}\cite[Theorem 4]{MoranTench}
\label{lem:MoranTench*} Let $p$ be a prime and let $G$ be a finite $p$-group. 
 Suppose that there exists a binary operation, $*$, on the set $G$ such that $G$ becomes an elementary abelian $p$-group under $*$.  Then there exists $N \lhd \left(\Z_p[G] \right)^\times$ such that $\left(\Z_p[G] \right)^\times/N \cong G$ if for all $a,b,c \in G$ the following two conditions hold:
 \begin{itemize}
  \item[(1)] $(c(a*b))*c = (ca)*(cb),$
  \item[(2)] $((a*b)c)*c = (ac)*(bc).$
 \end{itemize}
Moreover, \[ N = \Z_p^\times \times \left\{\sum_{g \in G} \lambda_g g \in \left(\Z_p[G]\right)^\times : \sideset{}{^*}\sum_{g \in G} \lambda_g^* g = 1\right\},\]
where $\sideset{}{^*}\sum$ denotes the summation in $G$ over $*$ and $\lambda_g^* g$ denotes $g * g * \cdots * g$ ($\lambda_g$ times).
\end{Lem}

Let $G$ be a $2$-group of nilpotency class $2$ and exponent $4$, and let $\{x_1, \dots, x_n\}$ be an ordered minimal generating set for $G$.  By Lemmas \ref{lem:nil2groups} and \ref{lem:nil2exp4}, every element of $G$ can be written uniquely as 
\begin{equation}\label{eq:repn}
 \prod_{i = 1}^n x_i^{k_{1,i}} \prod_{\{i : x_i^2 \neq 1\}} \left(x_i^2\right)^{k_{2,i}} \prod_{\{i,j : i < j, [x_i,x_j] \notin \{x_k, x_k^2, 1\}\}} [x_j, x_i]^{k_{3,i,j}},
\end{equation}
where each $k_{1,i}, k_{2,i},$ and $k_{3,i,j}$ is either $0$ or $1$.  We note that, among the $x_i$, $x_i^2$, and $[x_j, x_i]$, the only elements that can be noncentral are the $x_i$.  Using the representation of each group element listed in Equation \ref{eq:repn}, if 
\[ a = \prod_{i = 1}^n x_i^{k_{1,i}} \prod_{\{i : x_i^2 \neq 1\}} \left(x_i^2\right)^{k_{2,i}} \prod_{\{i,j : i < j, [x_i,x_j] \notin \{x_k, x_k^2, 1\}\}} [x_j, x_i]^{k_{3,i,j}}\]
and 
\[ b = \prod_{i = 1}^n x_i^{\ell_{1,i}} \prod_{\{i : x_i^2 \neq 1\}} \left(x_i^2\right)^{\ell_{2,i}} \prod_{\{i,j : i < j, [x_i,x_j] \notin \{x_k, x_k^2, 1\}\}} [x_j, x_i]^{\ell_{3,i,j}},\]
then we may define
\begin{equation}\label{eq:*}
a*b = \prod_{i = 1}^n x_i^{k_{1,i} \oplus \ell_{1,i}} \prod_{\{i : x_i^2 \neq 1\}} \left(x_i^2\right)^{k_{2,i} \oplus \ell_{1,i}} \prod_{\{i,j : i < j, [x_i,x_j] \notin \{x_k, x_k^2, 1\}\}} [x_j, x_i]^{k_{3,i,j} \oplus \ell_{3,i,j}},
\end{equation}
where $\oplus$ denotes addition modulo $2$.  For ease of notation, we will write $a$ as \\ $(k_{1,i}; k_{2,i}; k_{3,i,j})$, and, if we write $b$ as $(\ell_{1,i}; \ell_{2,i}; \ell_{3,i,j})$, then we may write
\[a*b = (k_{1,i} \oplus \ell_{1,i}; k_{2,i} \oplus \ell_{2,i}; k_{3,i,j} \oplus \ell_{3,i,j}).\]  
We will now show that the operation $*$ satisfies the hypotheses of Lemma \ref{lem:MoranTench*}.

\begin{Prop}
\label{prop:nil2exp4*}
 Let $G$ be a $2$-group with nilpotency class (at most) $2$ and exponent $4$ with a representation of elements as in Equation \ref{eq:repn}.  Then, the operation $*$ defined in Equation \ref{eq:*} makes $G$ into an elementary abelian $2$-group, and $*$ satisfies equations (1) and (2) of Lemma \ref{lem:MoranTench*} for all $a,b,c \in G$. 
\end{Prop}

\begin{proof}
 First, it is clear from the definition of $*$ that $(G,*)$ is elementary abelian.  We will check conditions (1) and (2) of Lemma \ref{lem:MoranTench*} in two steps.  To start, let $x_i$ have order $4$.  We will show that the elements in $\langle x_i \rangle$ satisfy conditions (1) and (2) of Lemma \ref{lem:MoranTench*}.  First, the equations will always hold if any of $a,b,$ or $c$ is $1$ or if $a = b$.  Second, since $\langle x_i \rangle$ is abelian, equation (1) holds if and only if equation (2) holds, and $c(a*b) = c(b*a)$, $(ca)*(cb) = (cb)*(ca)$.  This leaves exactly nine possibilities for $(a,b,c)$ to check: $(x_i, x_i^2, x_i)$, $(x_i, x_i^2, x_i^2)$, $(x_i, x_i^2, x_i(x_i)^2)$, $(x_i, x_i(x_i)^2, x_i)$, $(x_i, x_i(x_i)^2, x_i^2)$, $(x_i, x_i(x_i)^2, x_i(x_i)^2)$, $(x_i^2, x_i(x_i)^2, x_i)$, $(x_i^2, x_i(x_i)^2, x_i^2)$, and $(x_i^2, x_i(x_i)^2, x_i(x_i)^2)$.  These calculations are routine and left to the reader.

 Next, since (1) and (2) of Lemma \ref{lem:MoranTench*} hold for each $\langle x_i \rangle$ of order $4$, each $x_i$ that does not have order $4$ has order $2$, and all squares of elements are central (Lemma \ref{lem:nil2exp4}(ii)). Hence, to verify (1) and (2) of Lemma \ref{lem:MoranTench*}, it suffices to check that the exponents of the nontrivial commutators coincide.  Let $a = (k_{1,i}; k_{2,i}; k_{3,i,j})$, $b = (\ell_{1,i}; \ell_{2,i}; \ell_{3,i,j})$, and $c = (m_{1,i}; m_{2,i}; m_{3,i,j})$.  For (1), it is routine to verify that the exponent of the commutator $[x_j, x_i]$ for each of $(c(a*b))*c$ and $(ca)*(cb)$ is $k_{3,i,j} + \ell_{3,i,j} + (k_{1,i} + \ell_{1,i})m_{1,j}$. A similar calculation shows that the exponent of the commutator $[x_j, x_i]$ for each of $((a*b)*c)*c$ and $(ab)*(bc)$ is $k_{3,i,j} + \ell_{3,i,j} + (k_{1,j} + \ell_{1,j})m_{1,i}$.  Therefore, each of conditions (1) and (2) of Lemma \ref{lem:MoranTench*} hold for $G$ with the operation $*$ for all $a,b,c \in G$.
\end{proof}

Throughout the remainder of this section, $G$ will be a group of order $2^n$ and exponent 4. For an index set $M \subseteq \{1, \ldots, n\}$ and a collection $\{g_1, \ldots, g_n\}$ of indexed group elements, the notation $\prod_{m \in M} g_m$ means the product in the lexicographic order of $M$. That is, if $M = \{m_1, \dots, m_k\}$ where $m_1 < m_2 < \dots < m_k$, then $\prod_{m \in M} g_m = \prod_{i = 1}^k g_{m_i}$. Likewise, for a binary operation $*$ on $G$, let $\bigast_{m \in M} g_m$ denote the $*$ product of all $g_m$ for $m \in M$. As usual, $[b,a] = b^{-1}a^{-1}ba$ is the commutator of $a, b \in G$, and for three elements $a, b, c \in G$, their triple commutator is $[a, b, c] := [[a, b], c]$.

Next, we describe how we will construct a $*$ product on $G$. Assume $X=\{x_1, \ldots, x_n\}$ is a generating set for $G$ such that every element $g \in G$ can be expressed uniquely as
\begin{equation*}
g = \prod_{i=1}^n x_i^{\ep_i}
\end{equation*}
with each $\ep_i \in \{0, 1\}$. We define $*$ on $G$ by the following: if $g = \prod_{1=1}^n x_i^{e_i}$ and $h = \prod_{i=1}^n x_i^{f_i}$, then
\begin{equation}\label{eq:*op}
g \ast h := \prod_{i=1}^n x_i^{e_i + f_i \text{ mod 2}}.
\end{equation}
Under this operation, $G$ becomes an elementary abelian 2-group. It remains to determine conditions on $X$ for which $*$ satisfies parts (1) and (2) of Lemma \ref{lem:MoranTench*}.  We recall Definition \ref{def:goodseq}: a \textit{good generating sequence} $X = (x_1,\dots, x_n)$ is a generating set for $G$ such that every element $g \in G$ can be written uniquely as
\[ g = \prod_{i=1}^n x_i^{\ep_i},\]
where each $\ep_i \in \{0,1\}$, and the sequence further satisfies:
\begin{itemize}
    \item[(1)] for $i < j$, $[x_j, x_i] \in X \cup \{1\}$ and, if $[x_j, x_i] \notin Z(G)$, then:
        \begin{itemize}
            \item[(i)] $[x_j, x_i] = x_m$ for some $m \in \{1,\dots, n\}$ such that $i < m < j$;
            \item[(ii)] if $x_i^2 \notin Z(G)$, $x_i^2 = \prod_{m \in M} x_m$ for some $M \subseteq \{1,\ldots,n\}$, and $x_\ell \notin Z(G)$ for some $\ell \in M$, then $\ell < j$;
            \item[(iii)] if $x_j^2 \notin Z(G)$, $x_j^2 = \prod_{m \in M} x_m$ for some $M \subseteq \{1,\ldots,n\}$, and $x_\ell \notin Z(G)$ for some $\ell \in M$, then $i < \ell$;
        \end{itemize}
    \item[(2)] for $i < j < k$, if $[x_k, x_i, x_j] = 1$.  
\end{itemize}
Note that condition (2) implies that $[x_j, x_i, x_\ell] = 1$ for all $i < \ell < j$, so if $[x_j, x_i] = x_m$, then $[x_m, x_\ell] = 1$ for all $i < \ell < j$.  

For each $g \in G$, we refer to the unique expression $g = \prod_{i = 1}^t x_{k_i}$ as the \textit{normal form for $g$} (in terms of $X$).  So, conditions (1)(ii) and (iii) involve terms in the normal forms for $x_i^2$ and $x_j^2$, respectively.

We begin with a lemma regarding central elements of order $2$ and the $*$ product.

\begin{Lem}\label{lem:central}
Let $G$ be a group with good generating sequence $X = (x_1, \dots, x_n)$.  If $x_\ell \in Z(G)$ and $x_\ell^2 = 1$, then $g*x_\ell = gx_\ell$ for all $g \in G$.  In particular, if $G$ has exponent $4$ and nilpotency class $3$, then all triple commutators $[x_i, x_j, x_k] = x_\ell$ satisfy $gx_\ell = g*x_\ell$ for all $g \in G$.
\end{Lem}

\begin{proof}
 Let $g \in G$.  Since $X$ is a good generating sequence, we can express $g$ uniquely as 
 \[g = \prod_{m \in M} x_m.\]
 Let $x_\ell \in X$ be a central element of order $2$.  If $\ell \in M$, then
 \[ gx_\ell = \left(\prod_{m\in M, m < \ell}x_m\right) x_\ell^2 \left(\prod_{m \in M, m > \ell} x_m\right) = \prod_{m \in M, m \neq \ell} x_m = g*x_\ell,\]
 and, if $\ell \notin M$, then
 \[ gx_\ell = \prod_{m \in M \cup \{\ell\}} x_m = g*x_\ell.\]
Finally, in a group of exponent $4$ and nilpotency class $3$, all commutators have order $2$ and all triple commutators are central. 
\end{proof}

\begin{Lem}\label{lem:rewriting} 
Let $G$ be a group of exponent $4$ and nilpotency class $3$, and that has a good generating sequence $X = (x_1, \ldots, x_n)$.  For all $I \subseteq \{1, \dots, n\}$, we have:
 \begin{itemize}
  \item[(i)] $x_k \left(\prod_{i \in I}\limits x_i \right) = \left(\bigast_{i \in I \backslash \{k\}} x_kx_i\right) * x_k^{1 + \chi_{k \in I}} *x_k^{|I\backslash\{k\}| \bmod 2}$,
 \item[(ii)] $\left(\prod_{i \in I}\limits x_i \right)x_k = \left(\bigast_{i \in I \backslash \{k\}} x_ix_k\right) * x_k^{1 + \chi_{k \in I}} *x_k^{|I\backslash\{k\}| \bmod 2}$,
 \end{itemize}
where \[\chi_{k \in I} = 
\begin{cases} 0, \text{ if } k \notin I\\
1, \text{ if } k \in I
\end{cases}.
\]
\end{Lem}

\begin{proof}
We will prove (i). The proof of (ii) is similar. Let $x_{i,j} := [x_i, x_j]$ and $x_{i,j,k} := [x_i, x_j, x_k]$; each $x_{i,j}$ and $x_{i,j,k}$ is either an element of $X$ or the identity. With this notation, we have $x_k x_i = x_i x_{k,i} x_k x_{k, i, k}$ for all $i$ and $k$.

Since $G$ has nilpotency class $3$, all commutators have order $2$ and all triple commutators are central. Because each $x_{k,i,k}$ is central, we have
\begin{equation}\label{eq1}
x_k\left(\prod_{i \in I, i < k} x_i \right) = \left(\prod_{i \in I, i < k} x_i x_{k,i} x_{k,i,k} \right) x_k.
\end{equation}
This expression is almost in normal form, and we claim that no new terms are created by writing it in its normal form. Indeed, each triple commutator has order 2, so Lemma \ref{lem:central} shows that these terms can be placed anywhere without altering the $*$ product. Moreover, condition (1)(i) and (2) of Definition \ref{def:goodseq} guarantee that $x_{k,i}$ commutes with all $x_j$ such that $i < j < k$, so any $x_{k,i}$ can be moved to to put the expression into normal form without creating any new terms. Thus,
\begin{equation}\label{eq2}
\left(\prod_{i \in I, i < k} x_i x_{k,i} x_{k,i,k}\right)x_k = \left(\bigast_{i \in I, i < k} x_i x_{k,i} x_{k,i,k}\right)\ast x_k.
\end{equation}

Next, consider the product of the $x_i$ with $i \geq k$. By equation \eqref{eq1}, 
\begin{equation*}
x_k\left(\prod_{i \in I} x_i \right) = \left(\prod_{i \in I, i < k} x_i x_{k,i} x_{k,i,k} \right) x_k^{1 + \chi_{k \in I}}\left(\prod_{i \in I, i > k}\limits x_i \right).
\end{equation*}
The product $\prod_{i \in I, i > k} x_i$ is already in normal form.  Conditions (1)(ii) and (1)(iii) of Definition \ref{def:goodseq} guarantee that if $x_k^{1 + \chi_{k \in I}} = x_k^2$, then any terms involved in the normal form for $x_k^2$ can be moved to put the full expression into normal form without creating any new terms. Thus, by equation \eqref{eq2} (and keeping in mind that Lemma \ref{lem:central} applies to any central commutators), we obtain
\begin{equation*}
\left(\prod_{i \in I, i < k} x_i x_{k,i} x_{k,i,k} \right) x_k^{1 + \chi_{k \in I}}\left(\prod_{i \in I, i > k}\limits x_i \right) = \left(\bigast_{i \in I, i < k} x_i x_{k,i} x_{k,i,k} \right) \ast x_k^{1 + \chi_{k \in I}} \ast \left(\bigast_{i \in I, i > k}\limits x_i \right).
\end{equation*}
From here, we will insert additional copies of $x_k$ into the $*$ product and exploit the fact that $*$ is elementary abelian to get:
\begin{align*}
\left(\bigast_{i \in I, i < k} x_i x_{k,i} x_{k,i,k} \right) &\ast x_k^{1 + \chi_{k \in I}} \ast \left(\bigast_{i \in I, i > k}\limits x_i \right)\\
&= \left(\bigast_{i \in I, i < k}\limits x_i x_{k,i} x_{k,i,k}x_k\right)*x_k^{1 + \chi_{k \in I}}*\left(\bigast_{i \in I, i > k}\limits x_kx_i \right)*\bigast_{i \in I \backslash \{k\}} x_k\\
&= \left(\bigast_{i \in I, i < k}\limits x_kx_i\right)*x_k^{1 + \chi_{k \in I}}*\left(\bigast_{i \in I, i > k}\limits x_kx_i \right)*\bigast_{i \in I\backslash\{k\}} x_k\\
&= \left(\bigast_{i \in I \backslash \{k\}} x_kx_i\right) * x_k^{1 + \chi_{k \in I}} *x_k^{|I\backslash\{k\}| \bmod 2}.
\end{align*}
\end{proof}

Let $G$ have a good generating sequence $X$, and let $g \in G$ have normal form $g = \prod_{m \in M} x_m$, so $g$ is uniquely identified by $M \subseteq \{1, \dots, n\}$.  Define $\max_X(g) := \max\{m \in M\}$ (with the convention that $\max(1) = 0$), $\min_X(g) : = \min\{m \in M\}$,  and $L_X(g) := |M|$.  When there is no confusion as to what the good generating sequence is, we write $\max(g), \min(g)$, and $L(g)$, respectively.

\begin{Prop}\label{prop:main}
Let $G$ be a group of exponent $4$ and nilpotency class at most $3$, and that has a good generating sequence $X = (x_1, \dots, x_n)$.  Then, the operation $*$ defined in Equation \ref{eq:*op} makes $G$ into an elementary abelian $2$-group, and $*$ satisfies equations (1) and (2) of Lemma \ref{lem:MoranTench*} for all $a,b,c \in G$. 
\end{Prop}

\begin{proof}
We will prove that all $a,b,c \in G$ satisfy the conditions (1) of Lemma \ref{lem:MoranTench*}. The proof for condition (2) is similar.

We proceed by induction on $L(c)$. The result is trivial when $L(c) = 0$, i.e., when $c = 1$. Consider first the case when $L(c) = 1$, i.e., $c = x_k$.  Let $a = \prod_{m \in A} x_m$ and $b = \prod_{m \in B} x_m$, and, for sets $S$ and $T$, let $S \ominus T$ denote the symmetric difference of the two sets.  Then, 
\begin{equation*}
(c(a*b))*c = \left(x_k\left(\left(\prod_{m \in A} x_m\right)*\left(\prod_{m \in B} x_m\right)\right)\right)*x_k = \left(x_k\left(\prod_{m \in A \ominus B} x_m\right)\right)*x_k.
\end{equation*}
Applying Lemma \ref{lem:rewriting} with $I = A \ominus B$ yields
\begin{equation*}
\left(x_k\left(\prod_{m \in A \ominus B} x_m\right)\right)*x_k = \left(\bigast_{m \in (A \ominus B)\backslash \{k\}} x_kx_m\right) * x_k^{1 + \chi_{k \in A \ominus B}} *x_k^{|(A \ominus B)\backslash\{k\}| \bmod 2}*x_k.
\end{equation*}
Since $x_k^{1 +\chi_{k \in A \ominus B}}*x_k = x^{1 +\chi_{k \in A}} * x_k^{1 + \chi_{k \in  B}}$ and $|(A \ominus B) \backslash \{k\}| \bmod 2 = (|A \backslash \{k\}| + |B \backslash \{k\}| ) \bmod 2$, the last expression above is equal to
\begin{equation*}\label{eq:thm}
\left(\left(\bigast_{m \in A \backslash \{k\}} x_kx_m\right) * x_k^{1 + \chi_{k \in A}} *x_k^{|A\backslash\{k\}| \bmod 2}\right)*\left(\left(\bigast_{m \in B \backslash \{k\}} x_kx_m\right) * x_k^{1 + \chi_{k \in B}} *x_k^{|B\backslash\{k\}| \bmod 2} \right).
\end{equation*}
Using Lemma \ref{lem:rewriting} again produces
\begin{equation*}
\left(x_k\left(\prod_{m \in A} x_m\right)\right)*\left(x_k\left(\prod_{m \in B} x_m\right)\right),
\end{equation*}
which is equal to $(ca)*(cb)$. So, the result holds when $L(c)=1$.

Next, assume that $L(c) \ge 1$ and $(g(a*b)*g) = (ga)*(gb)$ holds for all $g, a, b \in G$ with $L(g) < L(c)$. Let $c = x_ud = x_u*d = ex_v = e*x_v$, where $u = \min(c)$ and $v = \max(c)$.  Again, let $a = \prod_{m \in A} x_m$ and $b = \prod_{m \in B} x_m$. Then,
\begin{align*}
(c(a*b))*c &= (x_u(d(a*b)))*(x_ud)\\
&= (x_u((d(a*b))*d))*x_u, \quad\text{ since $L(x_u)=1$}\\
&= (x_u((da)*(db)))*x_u, \quad\text{ since $L(d) < L(c)$}\\
&= (x_u(da))*(x_u(db)), \quad\text{ since $L(x_u)=1$}\\
&= (ca)*(cb),
\end{align*}
as desired.
\end{proof}

Computationally, all groups of exponent $4$, nilpotency class $3$, and order at most $64$ have a good generating sequence; $35$ out of $37$ groups of exponent $4$, nilpotency class $3$, and order $128$ have a good generating sequence (that is, all except \verb|SmallGroup(128,764)| and \verb|SmallGroup(128,836)|); $294$ out of $322$ groups of exponent $4$, nilpotency class $3$, and order $256$ have a good generating sequence; and $4046$ out of $4813$ groups of exponent $4$, nilpotency class $3$, and order $512$ have a good generating sequence.  

Propositions \ref{prop:nil2exp4*} and \ref{prop:main} prove that every group of exponent $4$ that either has nilpotency class at most $2$ or has nilpotency class $3$ and a good generating sequence has a normal complement in its mod $2$ envelope; again, we refer the curious reader to \cite{Ivory, Johnson, MoranTench} for more details.  We note that some of these normal subgroups appear to have been discovered previously (see \cite{PassiSehgal, Sandling}); however, as we see later in Example \ref{Norm comp, but not realizable}, a $2$-group having a normal complement in its mod $2$ envelope is not a sufficient condition for realizability in characteristic $2$. 

We can now prove Theorem \ref{thm:exp4nil2}, which shows that any 2-group of exponent 4 that either has nilpotency class $2$ or has nilpotency class $3$ and a good generating sequence is realizable in characteristic 2.

\begin{proof}[Proof of Theorem \ref{thm:exp4nil2}]
 By Propositions \ref{prop:nil2exp4*} and \ref{prop:main}, $G$ satisfies the hypotheses of Lemma \ref{lem:MoranTench*} with the binary operation $*$ as defined in either Equation \ref{eq:*} or \ref{eq:*op}.  By Lemma \ref{lem:MoranTench*}, the subgroup
 \[N := \left\{\sum_{g \in S} g  : S \subseteq G, |S| \text{ odd}, \sideset{}{^*}\sum_{g \in S} g = 1\right\}\]
 is normal in $\left(\Z_2[G]\right)^\times$ and $\left(\Z_2[G]\right)^\times/N \cong G$.  In order to show that $G$ is realizable, it suffices to show that $I := 1 + N$ is an ideal of $\Z_2[G]$.  Indeed, we have 
 \[ I = \left\{ \sum_{g \in S} g  : S \subseteq G, |S| \text{ even}, \sideset{}{^*}\sum_{g \in S} g = 1\right\}.\]  Since $*$ is an associative, commutative binary operation, it is clear that $I$ is closed under addition.  Moreover, using (1) and (2) of Lemma \ref{lem:MoranTench*} and proceeding by induction on $n$, if $g \in G$ and each $g_i \in G$, then
 \[\sideset{}{^*}\sum_{i=1}^n (g_ig) = \left(\sideset{}{^*}\sum_{i=1}^n g_i \right) * \left( \sideset{}{^*}\sum_{i= 1}^n g \right)\]
 and
 \[\sideset{}{^*}\sum_{i=1}^n (gg_i) = \left(\sideset{}{^*}\sum_{i=1}^n  g \right) * \left( \sideset{}{^*}\sum_{i= 1}^n g_i \right).\]  When $n$ is even and $x = \sum_{i=1}^n g_i \in I$, this shows that both $gx, xg \in I$.  Since $I$ is closed under addition, this shows that $I$ is an ideal, and, therefore, the group of units of $\Z_2[G]/I$ is isomorphic to $G$ by Lemma \ref{Norm comp lemma}(2).
\end{proof}

As was mentioned in the introduction, it is possible to show that many groups of order $2^n$ have exponent $4$ and nilpotency class $2$, and we will expand upon that further here.  Suppose that there are $2^{A(n) \cdot n^3}$ groups of order $2^n$.  Higman \cite{Higman} showed that $A(n) \ges 2/27 + O(n^{-1})$ by considering only groups of exponent $4$ and nilpotency class $2$.  On the other hand, Sims \cite{Sims} proved that $A(n) \les 2/27 + O(n^{-1/3})$, which indeed proves that 
\[ \frac{\log(\# \text{ of groups of order }2^n \text{ with exponent }4 \text{ and nilpotency class }2)}{\log(\# \text{ of groups of order } 2^n)} \to 1 \]
as $n \to \infty$.  It seems very likely that a large ratio of groups of order (at most) $2^n$ have exponent $4$ and nilpotency class $2$, perhaps even almost all as $n \to \infty$.  At any rate, our results prove that at least $2^{2n^3/27 + O(n^2)}$ out of the groups of order $2^n$ are realizable as groups of units of finite rings.

We end this section by noting that the condition that $G$ have exponent $4$ is necessary to the proof of Theorem \ref{thm:exp4nil2}. (It is in fact necessary to the statement as well; see Example \ref{Norm comp, but not realizable}.) To see what goes wrong with exponent at least $8$, suppose $x_1$ is a generator of $G$ of order $8$ and both $x_1^2$ and $x_1^4$ also appear in the list of generators.  In this case, any element of $\langle x_1 \rangle$ can be written in the form $x_1^{n_1}(x_1^2)^{n_2}(x_1^4)^{n_3}$, where each $n_i$ is either $0$ or $1$. Let $a = x_1$, $b = x_1(x_1^2)$, and $c = x_1$.  Then,
\[ (c(a*b))*c = \left(x_1\left(x_1*x_1(x_1^2)\right)\right)*x_1 = x_1^2,\]
whereas
\[(ca)*(cb) = (x_1^2)*(x_1^4) = (x_1^2)(x_1^4),\]
showing the necessity of exponent $4$ to the proofs of Propositions \ref{prop:nil2exp4*} and \ref{prop:main} and hence (more generally) to the proof of Theorem \ref{thm:exp4nil2}.  


\section{Not all 2-groups of exponent 4 are realizable in characteristic 2}
\label{sect:counterexample}

Here, we demonstrate that SmallGroup(64,34) in GAP \cite{GAP}, which has presentation
\begin{equation*}
\langle x_1, x_2 : x_1^4=x_2^2 = (x_1x_2)^4= (x_1^2x_2)^4 = 1 \rangle,
\end{equation*}
has exponent $4$, nilpotency class $2$, and is not realizable in characteristic 2.  The authors would like to thank Leo Margolis for calling our attention to this example. We remark that the strategy used below will also show that SmallGroup(64,35) (which also has exponent 4 and nilpotency class 4) is not realizable in characteristic 2.

Our technique is largely computational, and relies on calculations performed in GAP. Let $G = \text{SmallGroup(64,34)}$, $R = \Z_2[G]$, and suppose $I$ is an ideal of $R$ such that $(R/I)^\times \cong G$. We argue that such an $I$ cannot exist. To do this, we produce lists of elements of $R$ that must lie in such an $I$. This yields a small number of possible ideals $I$ and residue rings $R/I$ that are then analyzed in more detail.

The steps for our procedure are given below. We say that an element $x \in G$ \textit{survives} if its image in $R/I$ is nontrivial. In other words, $x$ survives when $1+x \notin I$. 

\begin{Lem}\label{lem:64}
Let $x \in G$ be such that $|x|=4$ and suppose every nontrivial element of $C_G(x)$ survives. Then, there exists $y \in C_G(x) \setminus \{1, x, x^2\}$ such that $y^2=x^2$ and $1+x+x^2+y \in I$.
\end{Lem}
\begin{proof}
Use a bar to denote passage from $R$ to $R/I$. The element $1+x+x^2$ is a unit of $R$, so $\overline{1+x+x^2}$ is a unit of $R/I$. Moreover, $1+x+x^2$ commutes with $x$, and $(1+x+x^2)^2=x^2$. Since the nontrivial elements of $C_G(x)$ survive, we conclude that $\overline{1+x+x^2} = \overline{y}$ for some $y \in C_G(x)$ with $y^2=x^2$. Finally, if $y \in \{1, x, x^2\}$, then either $x=1$ or $x^2=1$ in $R/I$. This contradicts that assumption that both $x$ and $x^2$ survive, so $y \notin \{1, x, x^2\}$.
\end{proof}

Next, we initialize this setting in GAP by running the following code snippet:
\begin{verbatim}
R:=GroupRing(GF(2), SmallGroup(64,34));; genR:=GeneratorsOfAlgebra(R);; 
one:=genR[1];; f1:=genR[2];; f2:=genR[3];; f3:=genR[4];; 
f4:=genR[5];; f5:=genR[6];; f6:=genR[7];; 
G:=Group([f1,f2,f3,f4,f5,f6]);;
\end{verbatim}
In our subsequent work, when we refer to \texttt{one}, \texttt{f1}, $\ldots$, \texttt{f6}, these are specific elements of $R$ as set up in the above code snippet.

\subsection*{Step 1: Show that every nontrivial element of $G$ survives.}

For each $x \in G$, we compute the ideal $I':= \langle 1+x\rangle \subseteq R$, and find $|R/I'|$. If $|R/I'| \le 64$, then $R/I'$ cannot realize $G$, so $1+x$ cannot be part of $I$. Below are the elements of $G$ for which $|R/I'| \ge 128$, along with the size of $|R/I'|$.

\begin{table}[h]
\begin{minipage}{0.33\linewidth}\centering
\begin{tabular}{l|c}
Element $x$ & $|R/\langle 1 + x \rangle|$\\
\hline
\texttt{f3} & 256\\
\texttt{f4} & 256\\
\texttt{f5} & 65536\\
\texttt{f6} & 4294967296\\
\texttt{f3*f4} & 256
\end{tabular}
\end{minipage}%
\begin{minipage}{0.33\linewidth}\centering
\begin{tabular}{l|c}
Element $x$ & $|R/\langle 1 + x \rangle|$\\
\hline
\texttt{f3*f5} & 256\\
\texttt{f3*f6} & 256\\
\texttt{f4*f5} & 256\\
\texttt{f4*f6} & 256\\
\texttt{f5*f6} & 65536
\end{tabular}
\end{minipage}%
\begin{minipage}{0.33\linewidth}\centering
\begin{tabular}{l|c}
Element $x$ & $|R/\langle 1 + x \rangle|$\\
\hline
\texttt{f3*f4*f5} & 256\\
\texttt{f3*f4*f6} & 256\\
\texttt{f3*f5*f6} & 256\\
\texttt{f4*f5*f6} & 256\\
\texttt{f3*f4*f5*f6} & 256
\end{tabular}
\end{minipage}
\caption{Some larger residue rings of $R$}
\label{1+x table}
\end{table}

For each $x$ in Table \ref{1+x table} such that $|R/\langle 1+x \rangle|=256$, we compute the nilpotency class of $(R/\langle 1 + x \rangle)^\times$. In each case, the nilpotency class is 1 or 2. Since $G$ has nilpotency class 4, we conclude that $G$ cannot be realized by a residue ring of any such $R/\langle 1 + x \rangle$. Similarly, one may calculate that $(R/\langle 1+x \rangle)^\times$ has nilpotency class 2 when $x=\texttt{f5}$ or $x=\texttt{f5*f6}$. Thus, every such $x$ survives, and it remains to show that \f{} survives.

Let $G_0 = G/\langle \f{} \rangle$, which is isomorphic to SmallGroup(32,6) in GAP. Then, 
$R_0 := R/\langle 1 + \f{} \rangle$ is isomorphic to $\Z_2[G_0]$. Suppose that $R_0$ contains a two-sided ideal $I_0$ such that $(R_0/I_0)^\times \cong G$. We will apply Lemma \ref{lem:64} to $R_0$ and $I_0$.

The group $G_0$ contains twenty elements of order 4, all of which survive in $R_0/I_0$ by our work above. By Lemma \ref{lem:64}, for each such $x \in G_0$ there exists $y \in C_{G_0}(x) \setminus \{1, x, x^2\}$ such that $y^2=x^2$ and $1+x+x^2+y \in I_0$. For each $x$, there are three choices for such a $y$, so in total there are sixty sums $1+x+x^2+y$ to consider. For each sum $1+x+x^2+y$, we form the principal ideal $\langle 1+x+x^2+y \rangle \subseteq R_0$ and then the residue ring $R_0/\langle 1+x+x^2+y \rangle$. We will show that $G$ cannot be realized as the unit group of a residue ring of any $R_0/\langle 1+x+x^2+y \rangle$.

A priori, we have sixty residue rings $R_0/\langle 1+x+x^2+y \rangle$ to consider. However, there are many instances where two different sums $1+x_1+x_1^2+y_1$ and $1+x_2+x_2^2+y_2$ generate the same ideal. After deleting duplicate ideals, we are left with only eight distinct ideals generated by a sum $1+x+x^2+y$.

Among these eight ideals, five produce residue rings of order at most 4096. For these five rings, the unit group of the residue ring and its nilpotency class can be computed in GAP. In each case, the nilpotency class is less than 4. The remaining three ideals yield residue rings of orders 65536, 65536, and 1048576. It may be possible to analyze the corresponding unit groups directly; however, this could be time-consuming. A more efficient method is to examine the corresponding ideals. For each of the three remaining ideals, there is at least one element $x$ such that the ideal does not contain a sum $1+x+x^2+y$. After including such a sum in the ideal, the corresponding residue ring became small enough that the nilpotency class of its unit group could be computed in GAP. In every case, the nilpotency class was at most 3. 

Putting everything together, we conclude that if $1+\f{} \in I$, then no residue ring of $R/I$ can realize $G$. Therefore, \f{} must survive.

\subsection*{Step 2: Focus on elements of order 4 and use Lemma \ref{lem:64}.}

The group $G$ contains 32 elements $x$ such that $|x|=4$ and $x^2 \notin \langle \texttt{f5}, \f \rangle$. For each $x$, the centralizer of $x$ in $G$ is
\begin{equation*}
C_G(x) = \{1, \,x, \,x^2, \,x^3, \,\f, \,x\f, \,x^2\f, \,x^3\f\},
\end{equation*}
and is isomorphic to $C_4 \times C_2$. By Step 1, every nontrivial element of $C_G(x)$ survives, so we may apply Lemma \ref{lem:64} to $x$. There are three possibilities for the element $y$ that satisfies the conclusion of Lemma \ref{lem:64}: $x^3, \; x^3\f, \; x\f$. Consequently, for each of the 32 elements $x$, one of the following three sums must be in $I$:
\begin{equation*}
1+x+x^2+x^3, \quad 1+x+x^2+x^3\f, \quad \text{ or } 1+x+x^2+x\f.
\end{equation*}

At this point, we have 32 elements $x \in G$ we are considering. However, note that if $1+x+x^2+g \in I$ for some $g \in G$, then by multiplying by $x^2$ we get $x^2+x^3+1+x^2g = 1 + x^3 + (x^3)^2 + x^2g \in I$. So, if know which sum $1+x+x^2+g$ is in $I$ for a given $x$, then we also know which sum is in $I$ for $x^3$. This reduces the number of possible $x$'s from 32 to 16.

We may use the following 16 elements of $G$ as the possibilities for $x$:
\begin{verbatim}
f1,  f1*f2,  f1*f3,  f1*f5,  f1*f6,  f1*f2*f3,  f1*f2*f4,
f1*f2*f5,  f1*f2*f6,  f1*f3*f4,  f1*f3*f6,  f1*f4*f5, 
f1*f2*f3*f6,  f1*f2*f4*f6,  f1*f2*f5*f6,  f1*f3*f4*f6
\end{verbatim}
For each of these sixteen $x$'s, consider the ideal $I'=\langle 1+x+x^2+x\f \rangle \subseteq R$. GAP confirms that $|R/I'|=256$, and that the unit group of $R/I'$ has nilpotency class 2. We conclude that for each $x$, $1+x+x^2+x\f$ cannot be in $I$.

\subsection*{Step 3: Eliminating the remaining cases}

At this point, we have the sixteen possibilities for $x$, and for each one $I$ must contain either $1+x+x^2+x^3$ or $1+x+x^2+x^3\f$. A priori, this gives us $2^{16}$ possible choices for $I$. However, we can further reduce the problem to eight cases, of which only two require a detailed analysis. Let $x_1 = \texttt{f1}$, $\quad x_2 = \texttt{f1*f2}, \quad$ and $\quad x_3 = \texttt{f1*f3}$ which are the first three element in the above list of possibilities for $x$. We will form eight ideals $I'$ by picking one of the two possible sums for each $x_i$, and letting $I'$ be the ideal of $R$ generated by those three sums. The table below summarizes these cases, along with the order of $R/I'$.

\begin{table}[h]
\begin{minipage}{\linewidth}\centering
{\renewcommand{\arraystretch}{1.4}
\begin{tabular}{c|c|c|c|c}
Case & Sum for $x_1$ & Sum for $x_2$ & Sum for $x_3$ & $|R/I'|$\\
\hline
1 & \asum{1} & \asum{2} & \asum{3} & 4096\\
\hline
2 & \asum{1} & \asum{2} & \bsum{3} & 1024\\
\hline
3 & \asum{1} & \bsum{2} & \asum{3} & 2048\\
\hline
4 & \asum{1} & \bsum{2} & \bsum{3} & 1024\\
\hline
5 & \bsum{1} & \asum{2} & \asum{3} & 1024\\
\hline
6 & \bsum{1} & \asum{2} & \bsum{3} & 2048\\
\hline
7 & \bsum{1} & \bsum{2} & \asum{3} & 1024\\
\hline
8 & \bsum{1} & \bsum{2} & \bsum{3} & 2048\\
\end{tabular}}
\caption{The final eight cases.}
\label{I' table}
\end{minipage}
\end{table}

For $G$ to be realized by $R/I$, at least one of these eight ideals $I'$ must be contained in $I$. In other words, we must be able to find at least one residue ring of one $R/I'$ that realizes $G$. If that happens, then $G$ is a quotient group of the unit group of $R/I'$. We will argue that this does not occur, and therefore that $G$ is not realizable.

\subsubsection*{Cases 2, 4, 5, and 7} 

GAP indicates that Cases 2, 4, 5, and 7 always produce the same ideal $I'$. Furthermore, in those cases the unit group of $R/I'$ has nilpotency class 3. Hence, $G$ cannot be realized in these cases.

\subsubsection*{Cases 3, 6, and 8}

GAP confirms that in Cases 3, 6, and 8, the residue rings have isomorphic unit groups. Fix one of these Cases, and let $U$ be the unit group of $R/I'$, which has order 1024. Then, $U$ has nilpotency class 4, so we cannot rule out this Case as we did in the previous step. Nevertheless, a more detailed examination will show that $G$ is not a quotient group of $U$.

\begin{itemize}
\item Suppose $N$ is a normal subgroup of $U$ such that $U/N \cong G$.

\item \texttt{GAP} indicates that $Z(U)$ has order 8.

Since $U$ surjects onto $G$, the image of $Z(U)$ is contained in $Z(G)$.

However, $|Z(G)|=2$, so either $Z(U) \subseteq N$, or $|Z(U) \cap N|=4$.

\item If $Z(U)$ were contained in $N$, then the nilpotency class of $U/N$ would be 3.

This cannot occur, so $Z(U) \not\subseteq N$.

Hence, $|Z(U) \cap N|=4$.

\item \texttt{GAP} indicates that $Z(U)$ has seven subgroups of order 4.

Each such subgroup $H$ is normal in $U$ (since $H$ is central), so we can look at $U/H$.

If $G$ occurs as a quotient group of $U$, then it also occurs as a quotient group of $U/H$.

\item Each quotient group $U/H$ has order 256. \texttt{GAP} says that these quotient groups fall into three isomorphism classes:
\begin{center}
\begin{verbatim}
[256, 1529],  [256, 1530],  [256, 3953],  [256, 5696] 
\end{verbatim}
\end{center}
The first three groups all have nilpotency class 3, so they can be ruled out.

The group \texttt{[256, 5696]} has nilpotency class 4, but another \texttt{GAP} calculation verifies that $G$ does not occur as a quotient group of \texttt{[256, 5696]}.
\end{itemize}

\subsubsection*{The Final Case}

By now, only Case 1 remains. The strategy here is the same as for Cases 3, 6, and 8. However, because the unit group of $R/I'$ is larger, the calculations take longer, and we don't have access to the $\texttt{IdGroup}$ command in GAP to examine quotient groups of $U$.

\begin{itemize}
\item Let $U = (R/I')^\times$, which is a group of order 2048.

As before, assume that $G$ occurs as a quotient group of $U$, and let $N \trianglelefteq U$ be such that $U/N \cong G$.

\item The center $Z(U)$ of $U$ has order 8.

As before, we must have $|Z(U) \cap N|=4$.

\item The center contains seven subgroups $H$ of order 4.

For each of these, form the quotient group $U/H$, which has order 512.

One of these quotient groups has nilpotency class 3, so we can discard it.

\item For each remaining $U/H$, find all of its quotient groups of order 64.

Many groups will occur, but they always fall into three isomorphism classes:
\begin{center}
\begin{verbatim}
[ 64, 32 ],  [ 64, 60 ],  [64, 90]
\end{verbatim}
\end{center}

None of these classes is \texttt{[64, 34]}, so we conclude that $G$ cannot be realized in Case 1.
\end{itemize}


\section{Groups of large exponent}\label{Large exponent section}
Our goal in this section is to prove Theorems \ref{Exponent upper bound} and \ref{Exp at least 8 cor 1}, which allow us to conclude that some 2-groups are not realizable in characteristic $2^m$. We shall prove Theorem \ref{Exponent upper bound} first by a straightforward counting argument.

\begin{Lem}\label{Factorial lem}
Let $m \ges 1$. Then, for all $1 \les k \les 2^m-1$, the product $\binom{2^m}{k}\cdot 2^k$ is divisible by $2^{m+1}$.
\end{Lem}
\begin{proof}
Let $v_2$ be the 2-adic valuation, i.e.,\ for all positive integers $n$, $v_2(n)$ equals the exponent of the largest power of 2 that divides $n$. Let $\ell = \lfloor \log_2(k) \rfloor$. As is well-known, Legendre's formula states that
\begin{equation*}
v_2(k!) = \sum_{i=1}^{\ell} \left\lfloor \frac{k}{2^i} \right\rfloor.
\end{equation*}
Clearly,
\begin{equation*}
\sum_{i=1}^{\ell} \left\lfloor \frac{k}{2^i} \right\rfloor \les \sum_{i=1}^{\ell} \frac{k}{2^i} \les k \sum_{i=1}^{\ell} \frac{1}{2^i} < k,
\end{equation*}
and since $v_2(k!)$ is an integer, we have $v_2(k!) \les k-1$. Thus,
\begin{equation*}
v_2( \textstyle\binom{2^m}{k}\cdot 2^k) \ges v_2(2^m) - v_2(k!) + v_2(2^k) \ges m - (k-1) + k = m+1,
\end{equation*}
as desired.
\end{proof}

\begin{Lem}\label{1+2t lem}
Let $m \ges 1$ and let $t$ be an indeterminate. Then, $(1+2t)^{2^{m-1}}=1$ in the polynomial ring $\Z_{2^m}[t]$.
\end{Lem}
\begin{proof}
Apply the Binomial Theorem and Lemma \ref{Factorial lem}.
\end{proof}

We are now able to prove Theorem \ref{Exponent upper bound}.

\begin{proof}[Proof of Theorem \ref{Exponent upper bound}]
Since $G$ is realizable in characteristic $2^m$, by Theorem \ref{R=2G thm} there is a residue ring $R$ of $\Z_{2^m}[G]$ such that $R^\times = G$, $R$ is local with maximal ideal $M$, and $M = 1+G$. 

The ideal $M$ is nilpotent; let $k$ be the smallest positive integer such that $M^k=\{0\}$. Then, for each $1 \les i \les k-1$, we have $M^i \supsetneqq M^{i+1}$. Since $|M|=|G|=2^n$, this implies that $k \les n+1$.

Next, fix $g \in G$. Then, $1+g \in M$, so $(1+g)^k = 0$; in fact, $(1+g)^\ell = 0$ for all $\ell \ges n+1$. Let $L=\lceil \log_2(n+1) \rceil$ as in the statement of the theorem. Then, $2^L$ is the smallest power of 2 greater than or equal to $n+1$, so $(1+g)^{2^L} = 0$. Note that $\binom{2^L}{i}$ is even for each $1 \les i \les 2^L-1$. Hence, applying the Binomial Theorem to $(1+g)^{2^L}$ shows that
\begin{equation*}
0 = (1+g)^{2^L} = 1 + g^{2^L} + 2\alpha
\end{equation*}
for some $\alpha \in R$. Rearranging this equation gives $g^{2^L} = -(1+2\alpha)$, and by Lemma \ref{1+2t lem},
\begin{equation*}
g^{2^{L+m-1}} = (g^{2^L})^{2^{m-1}} = (-1)^{2^{m-1}}(1+2\alpha)^{2^{m-1}} = 1.
\end{equation*}
Since $g \in G$ was arbitrary, we conclude that the exponent of $G$ is at most $2^{L+m-1}$.
\end{proof}

The proof of Theorem \ref{Exp at least 8 cor 1} is more complicated. We will first establish Theorem \ref{Higher exp thm}, which provides more restrictions on the exponents of 2-groups that are realizable in characteristic 2. One of these restrictions also holds in characteristic 4, and this is enough for us to prove Theorem \ref{Exp at least 8 cor 1}. 

We begin with a computational lemma.

\begin{Lem}\label{Polynomial lemma}
Let $t$ be an indeterminate, let $\PL \ges 2$, and let $k \ges 1$ be odd. Let $I$ be the ideal of the polynomial ring $\Z_2[t]$ generated by $1+t^{2^{\PL}}$ and $1+t+t^2+t^k$, and let $R=\Z_2[t]/I$. If $k \equiv 1 \mod 4$, then $t^2=1$ in $R$; and if $k \equiv 3 \mod 4$, then $t^4=1$ in $R$.
\end{Lem}
\begin{proof}
In what follows, all equations take place in $R$. Assume first that $k \equiv 1 \mod 4$. We will use induction on $\PL$. If $\PL=2$, then $t = t^k$ and hence $1+t^2 = 0$ in $R$. So, assume that $\PL \ges 3$ and that the lemma holds for $\PL-1$, i.e.\ that $t^{2^{\PL-1}}=1$ implies $t^2=1$. Since $k \equiv 1 \mod 4$, we have $(t^k)^{2^{\PL-2}} = t^{2^{\PL-2}}$. Then, $(1+t+t^2+t^k)^{2^{\PL-2}} = 1+t^{2^{\PL-1}}$, which means that $t^{2^{\PL-1}} = 1$ and thus $t^2=1$.

Now, assume that $k \equiv 3 \mod 4$. We again use induction on $\PL$. When $\PL=2$, we get $t^k=t^3$ and $1+t+t^2+t^3=0$. Multiplying the latter equation by $1+t$ results in $1+t^4=0$, and the lemma holds. So, assume that $\PL \ges 3$ and that $t^{2^{\PL-1}}=1$ implies $t^4=1$. For readability, let $y=t^{2^{\PL-3}}$. Since $k \equiv 3 \mod 4$, $(t^k)^{2^{\PL-3}}$ will equal either $y^3$ or $y^{-1}$ depending on whether $k$ is equivalent to 3 or 7 mod 8. Keeping this in mind, we see that $(1+t+t^2+t^k)^{2^{\PL-3}}$ equals either $1+y+y^2+y^3$ or $1+y+y^2+y^{-1}$. If $1+y+y^2+y^3=0$, then multiplication with $1+y$ results in $1+y^4=1+t^{2^{\PL-1}}=0$. On the other hand, if $1+y+y^2+y^{-1}=0$, then we obtain the same relation via multiplication with $y+y^2$. In either case, $t^{2^{\PL-1}}=1$, and hence $t^4=1$.
\end{proof}

\begin{Thm}\label{Higher exp thm}
Let $G$ be a finite $2$-group. For each $a \in G$, let $N_a \ges 0$ be such that the exponent of $C_G(a)/\langle a \rangle$ is $2^{N_a}$. That is, $N_a$ is the smallest non-negative integer such that $b^{2^{N_a}} \in \langle a \rangle$ for all $b \in C_G(a)$.

Assume that there exists $a \in G$ such that one of the following conditions holds:
\begin{enumerate}
\item[(i)] $N_a=0$ and $|a| \ges 8$.
\item[(ii)] $N_a=1$ and $|a| \ges 16$.
\item[(iii)] $N_a \ges 2$ and $|a| \ges 2^{2N_a+1}$.
\end{enumerate}
Then, $G$ is not realizable in characteristic 2.
\end{Thm}
\begin{proof}
We will prove (iii). The arguments for (i) and (ii) require only minor modifications, which we note at the end of the proof. 

We proceed by contradiction. If $G$ is realizable in characteristic $2$, then there is a residue ring $R$ of $\Z_{2}[G]$ such that $G = R^\times$. By Theorem \ref{R=2G thm}, $R$ is a local ring with maximal ideal $M$ and $M = 1 + G$. 

Fix $a \in G$ such that (iii) holds, and let $N=N_a$. Let $y \in \langle a \rangle$ such that $|y| = 2^{2N+1}$, and let $x=y^{2^N}$. Then, $|x| = 2^{N+1}$.

Consider $x^2+x$ and $y^2+y$. Both elements are in $M$, so there exist $g, h \in G$ such that $x^2+x=1+g$ and $y^2+y=1+h$. Notice that $g=x^2+x+1$, which commutes with $a$, so in fact $g \in C_G(a)$, and similarly for $h$. Next, we have
\begin{equation*}
1+h^{2^N} = (1+h)^{2^N} = (y^2+y)^{2^N} = (y^{2^N})^2+y^{2^N}=x^2+x=1+g.
\end{equation*}
Hence, $g=h^{2^N}$, which means that $g \in \langle a \rangle$. Moreover,
\begin{equation*}
g^{2^N} = (x^2+x+1)^{2^N} = x^{2^N},
\end{equation*}
so $g$ is an element of $\langle a \rangle$ such that $|g|=|x|=2^{N+1}$. Thus, $\langle g \rangle =  \langle x \rangle$ and $g = x^k$ for some odd integer $k$. Hence, we obtain $x^k=x^2+x+1$, or equivalently $1+x+x^2+x^k=0$. It follows that the subring of $R$ generated by $x$ is isomorphic to the ring $\Z_2[t]/I$ of Lemma \ref{Polynomial lemma} via the mapping $x \mapsto t$. By that lemma, $x^4=1$, which contradicts the fact that $|x|=2^{N+1}\ges 8$.

This proves (iii). For (i) or (ii), take $y \in \langle a \rangle$ such that $|y|=8$, and let $x=y^2$. Define $g$ and $h$ as before. Then, one may show that $g=x^3$ and $h \in \{y^3, y^7\}$. Lemma \ref{Polynomial lemma} may then be applied to $y$ to conclude that $|y|=4$, a contradiction.
\end{proof}

The lower bounds for $|a|$ in conditions (i) and (ii) in Theorem \ref{Higher exp thm} are the best possible. This is because the groups $C_8 \times C_2$ (with an element $a$ such that $|a|=8$ and $N_a=1$) and $C_{16} \times C_4 \times C_2 \times C_2$ (with an $a$ such that $|a|=16$ and $N_a=2$) are both realizable in characteristic 2 (see Example \ref{Direct factor not realizable}). Moreover, the conclusion of Theorem \ref{Higher exp thm} does not always hold in characteristic $2^m$ with $m \ges 2$. For instance, the group $C_{16} \times C_2$ satisfies (ii), and hence is not realizable in characteristic 2; however, $C_{16} \times C_2 \cong \Z_{64}^\times$, and so is realizable in characteristic 64. Fortunately, we are able to prove that condition (i) implies a group is not realizable in characteristic $2^m$, thus recovering a recent result of Chebolu and Lockridge \cite[Prop.\ 7.1]{ChebLockPGroup}.

\begin{customthm}{\ref{Exp at least 8 cor 1}}
Let $G$ be a finite $2$-group. Assume that there exists $a \in G$ such that $|a| \ges 8$ and $C_G(a) = \langle a \rangle$. Then, $G$ is not realizable in characteristic $2^m$ for any $m \ges 1$.
\end{customthm}
\begin{proof}
As in Theorem \ref{Higher exp thm}, if $G$ is realizable in characteristic $2^m$, there is a residue ring $R$ of $\Z_{2^m}[G]$ such that $R^\times = G=1+M$, where $M$ is the unique maximal ideal of $R$. Note that $\Z_{2^m}^\times$ is a central subgroup of $R^\times$, and so $\Z_{2^m}^\times \les C_G(a)$. By assumption, $C_G(a)$ is cyclic, so $\Z_{2^m}^\times$ is also cyclic. Hence, $m \les 2$, and $\rchar(R)$ is either 2 or 4. The case $\rchar(R)=2$ is ruled out by Theorem \ref{Higher exp thm}, so we will assume that $\rchar(R)=4$.

Let $x \in \langle a \rangle$ be such that $|x|=8$. Then, $x^4$ is the unique element of order 2 in $\langle a \rangle$. Observe that  $\Z_4^\times = \{1,-1\}$ is a subgroup of $C_G(a) = \langle a \rangle$. By the uniqueness of $x^4$, we must have $x^4 = -1$.

Now, for any $k$, we have $(2a^k-1)^2 = 1$. This means that $2a^k-1$ is a unit of $R$ that commutes with $a$. So, $2a^k - 1 \in \langle a \rangle$ and has order 1 or 2. If $|2a^k - 1| = 2$, then $2a^k - 1 = -1$, which implies that $2a^k = 0$; this contradicts the fact that $\rchar(R)=4$. Thus, $2a^k - 1 = 1$, and so
\begin{equation}\label{2a^k=2}
2a^k = 2 \text{ for all } k \ges 0.
\end{equation}

As in the characteristic 2 case, let $g \in G$ be such that $x^2+x=1+g$. Then, $g$ commutes with $a$, so $g \in \langle a \rangle$. Keeping in mind \eqref{2a^k=2}, on the one hand we have
\begin{equation*}
(x^2+x)^4 = x^8 + 2x^6 + x^4 = 1 + 2 - 1 = 2,
\end{equation*}
while on the other hand
\begin{equation*}
(1+g)^4 = 1 + 2g^2 + g^4 = -1 + g^4.
\end{equation*}
Thus, $-1+g^4 = 2$, which means that $g^4 = -1 = x^4$. Since $g \in \langle a \rangle$, we must have
\begin{equation*}
g \in \{x, x^3, x^5, x^7\} = \{x, x^3, -x, -x^3\}.
\end{equation*}

Suppose that $g = \pm x$. Then, $x^2+x = 1 \pm x$. Subtracting $x$ from both sides of this equation gives either $x^2=1$ or $x^2 = 1-2x=-1$. Both equations contradict the fact that $|x|=8$.

Next, suppose that $g= x^3$, so that $x^2+x=1 + x^3$. Multiplying both sides of the equation by $1 + x$ and simplifying produces $2=1+x^4$, which means that $x^4=1$. A similar contradiction is reached when $x^2+x=1-x^3$ after multiplication by $1+x$.

We reach a contradiction in all cases, so we conclude that $G$ is not realizable in characteristic 4.
\end{proof}

As a corollary, we recover a recent result of Chebolu and Lockridge \cite[Prop.\ 7.1]{ChebLockPGroup}.

\begin{Cor}\label{Exp at least 8 cor 2}
Let $G$ be a nonabelian group of order $2^n$, where $n \ges 4$. If $G$ has exponent $2^{n-1}$, then $G$ is not realizable as the group of units of a finite ring.
\end{Cor}
\begin{proof}
Assume that $G$ has exponent $2^{n-1}$, and let $a \in G$ with $|a| = 2^{n-1}$. Then, $\langle a \rangle$ is a maximal subgroup of $G$ contained in $C_G(a)$. If $C_G(a) = G$, then $a \in Z(G)$ and $[G:Z(G)] \les 2$. This implies that $G$ is abelian, which is a contradiction. So, $C_G(a) = \langle a \rangle$. By Proposition \ref{Ring char prop} and Theorem \ref{Exp at least 8 cor 1}, $G$ is not realizable.
\end{proof}

\begin{Ex}\label{Non realizable examples}
We exhibit some families of 2-groups to which Corollary \ref{Exp at least 8 cor 2} applies. (These examples, and others, can also be found in \cite{ChebLockPGroup}.)
\begin{enumerate}[(1)]
\item By Theorem \ref{Exp at least 8 cor 1}, for every $n \ges 3$ and for all $m \ges 1$, the cyclic group $C_{2^n}$ is not realizable in characteristic $2^m$. Of course, if $2^n+1$ is a prime, then $C_{2^n}$ is realizable in characteristic $2^n+1$, since $\F_{2^n+1}^\times \cong C_{2^n}$.

\item Recall that the generalized quaternion group $Q_{2^n}$ has presentation
\begin{equation*}
Q_{2^n} = \langle a, b : a^{2^{n-1}} = 1, a^{2^{n-2}}=b^2, bab^{-1} = a^{-1} \rangle.
\end{equation*}
The group $Q_{2^n}$ has exponent $2^{n-1}$ and is both nonabelian and indecomposable when $n \ges 3$. By Proposition \ref{Ring char prop} and Corollary \ref{Exp at least 8 cor 2}, $Q_{2^n}$ is not realizable when $n \ges 4$. Note, however, that the ordinary quaternion group $Q_8$ is realizable in characteristic 2, because $Q_8$ has exponent 4.

\item Similar to the last example, the quasidihedral group $QD_{2^n}$ of order $2^n$ has presentation
\begin{equation*}
QD_{2^n} = \langle a, b : a^{2^{n-1}} = b^2 = 1, bab^{-1} = a^{2^{n-2}-1} \rangle.
\end{equation*}
When $n \ges 4$, this group is also nonabelian, indecomposable, and has exponent $2^{n-1}$. Hence, it too is not realizable in these cases.
\end{enumerate}
\end{Ex}

\section{Intriguing examples and open questions}\label{Example section}

As we have seen, the realizability of 2-groups is not a simple matter. Factors that can affect the realization of a group $G$ as the unit group of the finite ring $R$ include the exponent of $G$, the nilpotency class of $G$, the characteristic of $R$, and whether or not $G$ has a normal complement in $(\Z_{m}[G])^\times$. In this final section, we have collected a number of examples and open questions related to these variables.  We begin with some examples that we find interesting, and we end with some questions we would like to see answered.

\begin{Ex}\label{Norm comp, but not realizable}
\textit{There exists a 2-group $G$ that has a normal complement in $(\Z_2[G])^\times$, but is not realizable in characteristic 2.} 
\end{Ex}

\begin{proof} The group $C_8$ is not realizable in characteristic 2 by Theorem \ref{Exp at least 8 cor 1}. However, calculations performed with GAP \cite{GAP} show that $(\Z_2[C_8])^\times \cong C_8 \times C_4 \times C_2 \times C_2$, and that the image of $C_8$ under the natural embedding $C_8 \to \Z_2[C_8]$ has multiple normal complements in $(\Z_2[C_8])^\times$.

For a nonabelian example, we can take $M_{16}$, the Modular or Isanowa group of order $16$, which is group SmallGroup(16,6) in GAP.  This group has presentation
\[ M_{16} = \left\langle x_1, x_2 : x_1^8 = x_2^2 = [x_2,x_1]^2 = x_1^4[x_2,x_1]= 1\right\rangle,\] 
order $16$, exponent $8$, and nilpotency class $2$. First, this group has a normal complement in $(\Z_2[M_{16}])^\times$ by \cite[Theorem 2]{Ivory}.  On the other hand, since $M_{16}$ is nonabelian and indecomposable, Proposition \ref{Ring char prop} tells us that $M_{16}$ can only be realized by $R$ if $\rchar(R)=2^m$ for some $m \ges 1$. However, $M_{16}$ has a self-centralizing subgroup $\langle x_1 \rangle$ of order $8$, so, by Theorem \ref{Exp at least 8 cor 1}, $M_{16}$ is not realizable in characteristic $2^m$ for any $m$; hence, $M_{16}$ is not realizable.
\end{proof}

\begin{Ex}
\textit{There exists a 2-group of exponent 8 and nilpotency class 2 that is not realizable.}
\end{Ex}

\begin{proof}
Once again, we can take $M_{16}$, as described in Example \ref{Norm comp, but not realizable}.
\end{proof}

\begin{Ex}\label{Exp 8, nilp 2, realizable}
\textit{There exists a 2-group of exponent 8 and nilpotency class 2 that is realizable.}
\end{Ex}

\begin{proof} Consider SmallGroup(32,37), which has presentation
\[ G = \left\langle x_1, x_2, x_3 : x_1^8 = x_2^2 = x_3^2 = x_1^4[x_2,x_1]= [x_3,x_1] = [x_3, x_2] = 1\right\rangle.\]
Then, $G$ has order $32$, exponent $8$, and nilpotency class $2$.  If we define the ideal $I$ by
\[ I := \langle 1 + x_1 + x_2 + x_1^5 x_2, \; 1 + x_1 + x_1^2 + x_1^7 x_3, \; 1 + x_1 + x_1^4 + x_1^5\rangle,\]
then the ring $\Z_2[G]/I$ has a group of units isomorphic to $G$, and hence at least some groups of exponent $8$ and nilpotency class $2$ are realizable.  What makes this example even more striking is that, if $M_{16}$ is the group SmallGroup(16,6) presented in Example \ref{Norm comp, but not realizable}, then $G \cong M_{16} \times C_2$, showing that, in certain cases, there are 2-groups that are realizable as a group of units even when not all of their direct factors are.
\end{proof}

\begin{Ex}
\textit{There exists a 2-group that is realizable in characteristic $2^m$ for some $m \ges 2$ but is not realizable in characteristic 2.}
\end{Ex}

\begin{proof} The group $C_{16} \times C_2$ is isomorphic to the unit group of $\Z_{64}$, and so is realizable in characteristic 64. However, $C_{16} \times C_2$ satisfies part (ii) of Theorem \ref{Higher exp thm}, and hence is not realizable in characteristic 2. More generally, the same is true for $\Z_{2^m}^\times \cong C_{2^{m-2}} \times C_2$ for all $m \ges 6$.
\end{proof}

\begin{Ex}\label{Exp 8, nonab, indecomp}
\textit{There exists an indecomposable nonabelian 2-group with exponent 8 that is realizable in characteristic 2.}
\end{Ex}
\begin{proof}
We give two examples, both of order 64. Let $G_1$ be SmallGroup(64,88), with presentation
\begin{equation*}
G_1 = \langle x_1, x_2, x_3 : x_1^8=x_2^2=x_3^2=[x_2,x_1]^2=[x_2,x_1^2]=x_1^4[x_3,x_1]=[x_3,x_2]=1\rangle
\end{equation*}
and let $G_2$ be SmallGroup(64,104), with presentation
\begin{equation*}
G_2 = \langle x_1, x_2, x_3 : x_1^8=x_2^4=x_3^2=x_2^2[x_2,x_1]=x_1^4[x_3,x_1]=[x_3,x_2]=1\rangle.
\end{equation*}
Then, $G_1$ and $G_2$ are realized as $(\Z_2[G_1]/I_1)^\times$ and $(\Z_2[G_2]/I_2)^\times$, respectively, where $I_1$ and $I_2$ are the ideals below:
\begin{align*}
I_1 &= \langle 1+x_1+x_1^2+x_1^3[x_2,x_1], \; 1+x_1+x_1x_2+x_2x_3, \; 1+x_1+x_1x_3+x_1^3x_3 \rangle,\\
I_2 &= \langle 1+x_1+x_1^2+x_1^3x_2^2, \; 1+x_1+x_1x_2+x_2x_3, \; 1+x_1+x_1x_3+x_1^4x_3 \rangle.
\end{align*}
\end{proof}

\begin{Ex}
\textit{There exists an indecomposable 2-group that is realizable in characteristic $2^m$ for some $m \ges 2$.}
\end{Ex}

\begin{proof} Following \cite[Part (F)]{Gilmer}, the ring 
\begin{equation*}
R = \Z[X]/\langle 4, 2X, X^2-2 \rangle \cong \Z_{4}[X]/\langle 2X, X^2+2 \rangle
\end{equation*}
has characteristic 4 and unit group isomorphic to $C_4$. Note that if $C_4 = \langle a \rangle$, then $R$ is also isomorphic to $\Z_4[C_4]/\langle 2a+a, a^2+1 \rangle$ via the mapping $X \mapsto 1+a$.

As for noncommutative examples, the dihedral group $D_8$ is realizable in characteristic 4 by \cite[Thm.\ 1.1]{ChebLockDihedral}. Also, the quaternion group 
\begin{equation*}
Q_8 = \langle \bfi, \bfj : \bfi^4=\bfj^4=1, \bfi^2=\bfj^2, \bfj\bfi\bfj^3 = \bfi^3 \rangle
\end{equation*}
is realized in characteristic 4 via $\Z_4[Q_8]/I$, where $I$ is the two-sided ideal
\begin{equation*}
I = \langle 2\bfi+2, \; 2\bfj+2, \; 1+\bfi+\bfi^2+\bfi^3, \; 1+\bfi+\bfj+\bfi\bfj \rangle.
\end{equation*}
\end{proof}

\begin{Ex}\label{Direct factor not realizable}
\textit{There exists a decomposable 2-group that is realizable with a direct factor that is not realizable.}
\end{Ex}

\begin{proof} The unit groups $\Z_{2^m}^\times \cong C_{2^{m-2}} \times C_2$ with $m \ges 6$ are examples, since $C_{2^{m-2}}$ need not be realizable. Other examples are also possible. In characteristic 2, neither $C_8$ nor $C_{16}$ is realizable. However, if $C_8 = \langle a \rangle$, then $C_8 \times C_2$ occurs as the unit group of $\Z_2[C_8]/\langle 1+a+a^4+a^5 \rangle$. Also, if $C_{16} = \langle a \rangle$, then $C_{16} \times C_4 \times C_2 \times C_2$ is the unit group of $\Z_2[C_{16}]/I$, where $I$ is the two-sided ideal
\begin{equation*}
I = \langle 1+a+\cdots+a^{15}, \; 1+a+a^8+a^9, \; 1+a^2+a^8+a^{10} \rangle.
\end{equation*}
More generally, \cite[Prop.\ 4.8]{DelCorsoDvornFiniteGroups} shows that if $G$ is a finite abelian 2-group of exponent $2^k$, where $k \ges 2$, then for all $m \ges k-2$, the group $G \times C_{2^m} \times C_2$ is realizable in characteristic $2^{m+1}$.

Lastly, for a nonabelian example, we return again to $M_{16}$. The group $M_{16}$ is not realizable, but $M_{16} \times C_2$ (SmallGroup(32,37) from Example \ref{Exp 8, nilp 2, realizable}) is realizable in characteristic 2.
\end{proof} 

Finally, we end with some open questions.



\begin{Ques}
 Is there a $2$-group with exponent $4$ and nilpotency class $3$ that is not realizable in characteristic $2$?  Is there a $2$-group realizable in characteristic $2$ with nilpotency class at least $4$?  
\end{Ques}

Thus far, we have seen that many groups with exponent $4$ and nilpotency class $3$ are realizable in characteristic $2$ (Theorem \ref{thm:exp4nil2}), whereas we saw in Section \ref{sect:counterexample} that SmallGroup(64,34) -- which has exponent $4$ and nilpotency class $4$ -- is not realizable in characteristic $2$.  It would be extremely interesting to know precisely where the dividing line between groups of exponent $4$ that are realizable and groups of exponent $4$ that are not realizable lies.  Moreover, as we saw in the proof of Proposition \ref{prop:main} (see in particular the proof of Lemma \ref{lem:rewriting}), having an elementary abelian commutator subgroup is crucial for the application of Moran and Tench's realizability criterion (Lemma \ref{lem:MoranTench*}).  Since $2$-groups with nilpotency class $4$ and larger no longer necessarily have elementary abelian commutator subgroups, it would be quite interesting if such a group were in fact realizable in characteristic $2$.

\begin{Ques}
For a given integer $m \ges 4$, is there a nonabelian, indecomposable $2$-group of exponent $2^m$ that is realizable in characteristic $2$?  If such a $2$-group exists for all $m$, then what is the behavior of $f(m)$, where $2^{f(m)}$ is the smallest order of such a $2$-group?  
\end{Ques}

We have seen that all groups of exponent $4$ are realizable, and the smallest nonabelian group of exponent $2^2 = 4$ has order $2^3 = 8$, so $f(2) = 3$.  Moreover, direct calculation using GAP shows that there are no nonabelian, indecomposable groups of exponent $8$ and order $16$ or $32$ that are realizable in characteristic $2$, whereas Example \ref{Exp 8, nonab, indecomp} shows that there do exist nonabelian, indecomposable groups of exponent $8$ and order $64$ that are realizable in characteristic $2$, so $f(3) = 6$.  Furthermore, Example \ref{Exp 8, nonab, indecomp} provides some evidence that perhaps there exist nonabelian, indecomposable groups with larger exponents that are realizable in characteristic $2$ at large enough orders, although it is still an open question as to whether there exists a nonabelian, indecomposable group of exponent $16$ or greater that is realizable in characteristic $2$.

\begin{Ques}\label{Realizable, but not in char 2}
Does there exist an indecomposable 2-group that is realizable in characteristic $2^m$ for some $m \ges 2$, but is not realizable in characteristic 2?
\end{Ques}

Let $G$ be such an indecomposable 2-group. If $G$ is abelian, then $G$ is cyclic, and hence to be realizable in characteristic $2^m$, $G$ must be either $C_2$ or $C_4$, both of which are realizable in characteristic 2. Moreover, if $G$ has exponent 4, then $G$ is realizable in characteristic 2 by Theorem \ref{thm:exp4nil2}. Thus, if such a $G$ exists, it must be nonabelian and have exponent at least 8.

\begin{Ques}
Let $G$ be a nonabelian 2-group that is not realizable. Does there exist a 2-group $H$ such that $G \times H$ is realizable?
\end{Ques}

This question is inspired by the situation with $C_{2^n}$ (for $n \ges 3$) and $M_{16}$. These groups are not realizable in characteristic $2^m$, but they  become realizable after attaching a direct factor of $C_2$. If $G$ is abelian, then such an $H$ always exists \cite[Prop.\ 4.8]{DelCorsoDvornFiniteGroups}, but the question is open in the case where $G$ is nonabelian.

\subsection*{Acknowledgements}  The authors would like to thank Steve Glasby for pointing out references \cite{Higman} and \cite{Sims} and Leo Margolis for bringing to our attention the example in Section \ref{sect:counterexample} as well as references \cite{PassiSehgal} and \cite{Sandling}.


\end{document}